\documentclass[11pt]{amsart}

\usepackage{mathptmx}
\usepackage{amsmath,amssymb,amsthm,mathrsfs,enumitem}
\usepackage[colorlinks=true,linkcolor=black,citecolor=red,urlcolor=blue]{hyperref}

\newtheorem{theorem}{Theorem}[section]
\newtheorem{proposition}[theorem]{Proposition}
\newtheorem{corollary}[theorem]{Corollary}
\newtheorem{lemma}[theorem]{Lemma}
\newtheorem{remark}[theorem]{Remark}

\newtheorem{assumption}[theorem]{Assumption}

\numberwithin{equation}{section}

\DeclareMathOperator{\Ker}{Ker}
\DeclareMathOperator{\gr}{gr}
\DeclareMathOperator{\spanC}{span}
\DeclareMathOperator{\ord}{ord}

\selectfont

\begin{document}

% \title[Weighted composition operators]{Weighted composition operators on functional quasi-Banach spaces: dynamical rigidity and affine symbols}
\title[Dynamical rigidity for weighted composition operators]{Dynamical rigidity for weighted composition operators on holomorphic function spaces}

\author{Isao Ishikawa}
\address{Kyoto University, Kyoto, Japan}
\email{ishikawa.isao.5s@kyoto-u.ac.jp}

\date{}

\begin{abstract}
We study weighted composition operators on quasi-Banach spaces of holomorphic functions via their induced action on jets along periodic orbits. Under a natural graded nondegeneracy condition, boundedness and compactness, together with a nonvanishing condition on the weight along the periodic orbit, impose strong restrictions on the local holomorphic dynamics of the symbol. We also obtain local periodic-point obstructions from supercyclicity, hypercyclicity, and cyclicity.
As consequences, we obtain affine-symbol rigidity for bounded weighted composition operators on spaces of entire functions. In one complex variable, if the ambient function space is any infinite-dimensional quasi-Banach space continuously embedded in the space of entire functions, then boundedness forces the symbol to be affine. In particular, this applies to every infinite-dimensional reproducing kernel Hilbert space of entire functions. We also prove a higher-dimensional affine-rigidity theorem under mild stability assumptions, and a weighted rigidity theorem for polynomial automorphisms of two complex variables.  Our approach relies on local holomorphic dynamics at periodic points rather than reproducing-kernel formulas or space-specific norm estimates, and it applies uniformly across broad classes of holomorphic function spaces.
\end{abstract}

\maketitle

\section{Introduction}\label{sec:introduction}
% In this paper, we develop a theory of weighted composition operators that leads to new rigidity results and dynamical obstructions.
% Let $X$ be a connected complex manifold and let $\mathcal{O}$ be the sheaf of holomorphic functions on $X$.
% We denote by $\mathcal{O}(X)$ the space of global holomorphic functions on $X$ and equip it with the topology of uniform convergence on compact sets.
% We also denote the germ at $p$ by $\mathcal{O}_{X,p}$, and the maximal ideal of germs vanishing at $p$ by $\mathfrak{m}_p$.
In this paper, we study to what extent operator-theoretic properties of a weighted composition operator
\[
uC_f:h\mapsto u\cdot(h\circ f)
\]
force rigidity of the underlying holomorphic map $f$. 
Our main result is that boundedness, compactness, and linear-dynamical behavior of $uC_f$ impose strong restrictions on the local holomorphic dynamics of $f$ at periodic points.

The basic tool is a jet filtration along periodic orbits. 
It produces finite-dimensional graded structures on which the weighted composition operator acts, and these structure have enough local dynamical information to detect obstructions. 
Under a natural graded nondegeneracy hypothesis on the ambient function space, we obtain local restrictions from boundedness and compactness, and also from supercyclicity, hypercyclicity, and cyclicity.

This local theory has several global consequences. In one complex variable, boundedness of a weighted composition operator with nonzero weight on any infinite-dimensional quasi-Banach space continuously embedded in $\mathcal O(\mathbb C)$ forces the symbol to be affine. In higher dimensions, we obtain an affine-rigidity theorem under mild stability assumptions, and for polynomial automorphisms of $\mathbb C^2$ we prove a weighted rigidity theorem. 
Thus the paper is not only about operator-theoretic properties of specific function spaces but it gives a general mechanism that turns operator assumptions into rigidity statements for holomorphic dynamical systems.

Let $V\subset \mathcal{O}(X)$ be a linear subspace.
In this paper, we always assume that the inclusion 
\[\iota: V \hookrightarrow \mathcal{O}(X)\] 
is continuous.
For almost all results in this paper, we assume that $V$ is a quasi-Banach
space. Recall that a quasi-Banach space is a complete Hausdorff topological
vector space endowed with a quasi-norm $\|\cdot\|_V$, that is, a map
\[
\|\cdot\|_V\colon V\to \mathbb{R}_{\ge 0}
\]
such that for all $a\in\mathbb{C}$ and $v,w\in V$,
\begin{enumerate}[label={\rm (\arabic*)}]
\item $\|av\|_V=|a|\,\|v\|_V$,
\item $\|v+w\|_V\le K(\|v\|_V+\|w\|_V)$ for some $K \ge 1$ independent of $v$ and $w$,
\item $\|v\|_V=0$ implies $v=0$.
\end{enumerate}
% A quasi-Banach space is also characterized as a complete Hausdorff topological vector space having a bounded neighborhood of $0$ (for the details, see \cite{kalton_peck_roberts_1984}).
A typical example of a quasi-Banach space with continuous inclusion is a reproducing kernel Hilbert space of holomorphic functions (see, for example, \cite[Section~3]{ISHIKAWA2023109048}).

Let $f:X\to X$ be holomorphic, and let $u\in \mathcal{O}(X)$.
We define the weighted pull-back $uf^*:\mathcal{O}(X)\to \mathcal{O}(X)$ by
\[
uf^*(h)=u\cdot (h\circ f)
\qquad (h\in \mathcal{O}(X)).
\]
We note that $uf^*$ induces a continuous linear map from $\mathcal{O}_{X,f(p)}$ to $\mathcal{O}_{X,p}$ for each $p\in X$.
We denote by $uC_f$ the linear operator on $V$ as the restriction of $uf^*$ to $V$, with domain $D(uC_f):=\{h\in V: u\cdot(h\circ f)\in V\}$.
When $u \equiv 1$, we omit $u$ from the notation and denote the operators by $f^*$ and $C_f$ for simplicity.

Our aim is to connect boundedness, compactness, and (super, hyper)cyclicity of $uC_f$ with the dynamical properties of $f$.
We show that bounded weighted composition operators with nonvanishing weight force strong restrictions on the symbol, and in many cases force the symbol to be affine.
Also, we show that compactness and (super, hyper)cyclicity give further dynamical obstructions.

Let us introduce the jet filtration at a point, which is the main tool in our analysis.
Then, we can show that the weighted pull-back $uf^*$ induces a well-defined map
\begin{align}
    \gr_{p}^n(uf^*):
    \mathfrak{m}_{f(p)}^n\big/\mathfrak{m}_{f(p)}^{n+1}
    \longrightarrow
    \mathfrak{m}_p^n\big/\mathfrak{m}_p^{n+1}.
    \label{eq:local-graded-map}
\end{align}
Let $\iota_p: V\hookrightarrow \mathcal{O}(X)\to \mathcal{O}_{X,p}$ be the natural map, and for each $n \ge 0$, define a closed subspace
\[
V_{p,n}:=\iota_p^{-1}(\mathfrak{m}_p^n).
\]
In other words, $V_{p,n}$ is the space of functions in $V$ whose Taylor coefficients at $p$ of total degree less than $n$ all vanish.
Via the inclusion $\iota:V\hookrightarrow \mathcal{O}(X)$, we naturally regard $V_{p,n}/V_{p,n+1}$ as a subspace of $\mathfrak{m}_p^n\big/\mathfrak{m}_p^{n+1}$.

\begin{assumption}
\label{ass:image-condition}
The tuple $(V,f,u,p)$ satisfies
\[
{\rm Im}\,\gr_{p}^n(uf^*) \subset V_{p,n}/V_{p,n+1}
\]
for infinitely many $n\ge 1$.
\end{assumption}

\begin{assumption}
\label{ass:image-condition-strong}
The pair $(V,p)$ satisfies
\[
\mathfrak{m}_p^n\big/\mathfrak{m}_p^{n+1} = V_{p,n}/V_{p,n+1}
\]
for infinitely many $n\ge 0$.
\end{assumption}

Obviously, Assumption \ref{ass:image-condition-strong} for $(V,p)$ implies Assumption \ref{ass:image-condition} for $(V,f,u,p)$ for all $f$ and $u$.
Typical function spaces $V$ satisfy Assumption \ref{ass:image-condition-strong} for sufficiently many $p$.
Several concrete examples are listed in \cite[Section~3]{ISHIKAWA2023109048}, including many RKHS and Fock-type spaces.
We also emphasize that this condition is equivalent to the kernel condition for the map ``$\kappa_p^n$'' introduced in \cite{ISHIKAWA2023109048} (see Appendix~\ref{app:comparison}).

For a periodic point $p$ of period $r$, we write
\[
u_r:=\prod_{j=0}^{r-1}u\circ f^j.
\]
Then, we state the main theorem:
\begin{theorem}[boundedness and compactness]\label{thm:bounded-local}
Let $V\subset \mathcal{O}(X)$ be a quasi-Banach space with continuous inclusion.
Let $p\in X$ be a periodic point of $f$ with period $r\ge 1$.
Assume that $uC_f$ is a bounded (resp. compact) linear operator on $V$, that $u_r(p)\neq 0$, and that $(V,f^r,u_r,p)$ satisfies Assumption~\ref{ass:image-condition}.
Then, every eigenvalue $\alpha$ of the Jacobian matrix ${\rm d}(f^r)_p:T^{1,0}_pX\to T^{1,0}_pX$ satisfies $|\alpha|\le 1$ (resp. $|\alpha|<1$).
\end{theorem}

Let $\mathrm{SU}(d)$ be the set of unitary matrices of size $d$ with determinant 1.
Combining Theorem~\ref{thm:bounded-local} with a fact on non-affine holomorphic maps, we obtain a general affine-rigidity result:
\begin{theorem}[affine rigidity from boundedness]\label{thm:global-affine}
Let $X = \mathbb{C}^d$.
Let $V\subset \mathcal{O}(\mathbb{C}^d)$ be a quasi-Banach space with continuous inclusion.
Assume that the following two conditions hold:
\begin{enumerate}[label={\rm (\arabic*)}]
\item Assumption~\ref{ass:image-condition-strong} for $(V,p)$ holds for every $p \in \mathbb{C}^d$,
\item for every $a\in(0,1)$ and $U \in \mathrm{SU}(d)$, there exists a nonvanishing function $v\in \mathcal{O}(\mathbb{C}^d)$ such that $vC_{aU}$ is bounded on $V$.
\end{enumerate}
Let $f:\mathbb{C}^d\to \mathbb{C}^d$ be holomorphic, and let $u\in \mathcal{O}(\mathbb{C}^d)$ be nonvanishing.
If $uC_f$ is bounded on $V$, then $f$ is affine.
\end{theorem}

Our approach is conceptually different from most earlier affine-symbol theorems for concrete spaces of entire functions.
In those works, affineness is typically derived from explicit formulas for reproducing kernels, coefficient weights, or other estimates that work for a specific function space; see, for example, \cite{CMS,Le2014,DKL17,SS17,IIS20,CARROLL2021125234}.
In contrast, Theorem~\ref{thm:global-affine} derives affineness from a dynamical viewpoint, namely the boundedness problem is linked to the existence of local dynamical obstructions for non-affine holomorphic maps, rather than to explicit analysis of a particular function space.

\medskip
Our approach also provides a necessary condition for cyclicity of weighted composition operators from the viewpoint of local dynamics.
For a topological vector space $W$ and a densely defined linear operator $S$ on $W$ with domain $D(S)$, let
\[
D(S^0):=W,\qquad D(S^{n+1}):=\{x\in D(S^n):S^n x\in D(S)\}\quad (n\ge0),
\]
and set
\[
D_\infty(S):=\bigcap_{n=0}^\infty D(S^n).
\]
We call $S$ \emph{hypercyclic} if there exists $x\in D_\infty(S)$ such that $\{S^n x:n\ge 0\}$ is dense in $W$, and we call such a vector $x$ a hypercyclic vector for $S$.
We call $S$ \emph{supercyclic} if there exists $x\in D_\infty(S)$ such that $\{\lambda S^n x:\lambda\in\mathbb{C},\ n\ge 0\}$ is dense in $W$, and we call such a vector $x$ a supercyclic vector for $S$.
We call $S$ \emph{cyclic} if there exists $x\in D_\infty(S)$ such that $\spanC\{S^n x:n\ge 0\}$ is dense in $W$, and we call such a vector $x$ a cyclic vector for $S$.

Then, we have a result on supercyclicity and hypercyclicity:
\begin{theorem}[supercyclicity and hypercyclicity]\label{thm:hyper-local}
Let $V\subset \mathcal{O}(X)$ be a general complex topological vector space with continuous inclusion.
Assume that $uC_f$ is supercyclic (resp. hypercyclic) on $V$ with $\dim V \ge 2$ (resp. $\dim V \ge 1$).
Then, $f$ has no periodic points.
\end{theorem}
Let $\delta_p: \mathcal{O}(X)\to \mathbb{C};~h\mapsto h(p)$ be the evaluation functional at $p$ and let $\mathcal{P}_r(f) := \bigl\{p\in X : f^r(p)=p\bigr\}$ be the set of periodic points of $f$ whose period divides $r$.
Then, we have the following result on cyclicity:
\begin{theorem}[cyclicity]\label{thm:cyclic-local}
Let $V\subset \mathcal{O}(X)$ be a complex topological vector space with continuous inclusion.
Assume that $\{\delta_p|_V : p \in X\}$ is linearly independent in the continuous dual space $V'$.
Assume that $uC_f$ is cyclic on $V$.
Then, for every $r\ge 1$ and every $\lambda\in\mathbb{C}$,
\[
\#\bigl(\mathcal{P}_r(f) \cap u_r^{-1}(\lambda)\bigr)\le r.
\]
In particular, in the unweighted case ($u \equiv 1$), for $r \ge 1$, we have
\[
\#\mathcal{P}_r(f)\le r
\]
% hence there exists at most one periodic orbit of exact length $r$ for each $r\ge 1$.
\end{theorem}

We note that when $V$ is a Hilbert space, hence a reproducing kernel Hilbert space
of holomorphic functions, the linear independence of
$\{\delta_p|_V : p \in X\}$ is equivalent to the associated reproducing kernel
being strictly positive definite.

\medskip
In the one-variable case ($\dim X=1$), we can show that
Assumption~\ref{ass:image-condition-strong} for $(V,p)$ automatically holds for every $p$ if $V$ is infinite-dimensional (see Lemma~\ref{lem:one-dim-graded});
in particular, Assumption~\ref{ass:image-condition} holds for $(V,f,u,p)$ for any $f$, $u$, and $p$.
Using the crucial properties of one-variable holomorphic functions, we have the following theorem:
\begin{theorem}\label{thm:repelling-fixed-point-strong}
Assume that $\dim_{\mathbb C}X=1$.
Let $V\subset \mathcal{O}(X)$ be an infinite-dimensional quasi-Banach space
with continuous inclusion.
Let $f:X \to X$ be holomorphic and let $u\in\mathcal{O}(X)\setminus\{0\}$.
Assume that $uC_f$ is bounded on $V$.
Then, for every fixed point $p \in X$ of $f$, we have
\[
|f'(p)|\le 1.
\]
\end{theorem}

As a result, we completely determine the possible symbols of bounded weighted composition operators on this general class of quasi-Banach spaces of entire functions.
\begin{corollary}\label{cor:repelling-fixed-point-strong-entire}
Let $X=\mathbb{C}$.
Let $V\subset \mathcal{O}(\mathbb{C})$ be an infinite-dimensional quasi-Banach space
with continuous inclusion.
Let $f:\mathbb{C}\to\mathbb{C}$ be entire, let $u\in\mathcal{O}(\mathbb{C})\setminus\{0\}$,
and assume that $uC_f$ is bounded on $V$.
Then, there exists $a, b \in \mathbb{C}$ with $|a| \le 1$ such that $f(z) = az + b$.
\end{corollary}

We note that we may take any infinite-dimensional reproducing kernel Hilbert space composed of entire functions as $V$, since it is always continuously included in the space of entire functions equipped with compact-open topology.

This theorem is a weighted counterpart of the affine-rigidity obtained in \cite{ISHIKAWA2023109048}.
We note that this affine-rigidity theorem holds even if $V$ is finite-dimensional in the unweighted case \cite[Theorem 1.2]{ISHIKAWA2023109048}.
However, the infinite-dimensionality is necessary in the general weighted case, for example, consider $V=\mathbb{C}e^{-z}$, $u = e^{z^2/2}$, and $f(z) = (z+1)^2/2$.

Weighted composition operators on spaces of entire functions have been extensively studied on concrete Fock-type and weighted Banach spaces; see, for example, \cite{Ueki2007,Ueki2010,Le2014,HaiKhoi2016,TienKhoi2019,ArroussiTong2019,BeltranMeneu2020,BeltranMeneuJorda2021,HaiRosenfeld2021,Bonet2022Survey}. 
Our methodology is different. 
Rather than seeking space-specific boundedness criteria, we develop a local dynamical mechanism that yields affine-symbol rigidity under very soft assumptions on the holomorphic function space.

\medskip
In the case of $X=\mathbb{C}^2$, we obtain a rigidity result for weighted composition operators whose symbols are polynomial automorphisms.
Let ${\rm M}_2(\mathbb{C})$ be the space of complex matrices of size $2$ and let ${\rm GL}_2(\mathbb{C})$ be the set of invertible matrices.
We define
\[
\mathcal{G}_2(V):=
\Big\{A\in {\rm GL}_2(\mathbb{C}) : \exists b\in \mathbb{C}^2,\ \exists w\in \mathcal{O}(\mathbb{C}^2),\ 
e^wC_{A(\cdot)+b}\text{ bounded on }V\Big\}.
\]
This is a subsemigroup of ${\rm GL}_2(\mathbb{C})$.

\begin{theorem}\label{thm:C2}
Let $X = \mathbb{C}^2$.
Let $V\subset \mathcal{O}(\mathbb{C}^2)$ be a quasi-Banach space with continuous inclusion.
Assume that the following two conditions hold:
\begin{enumerate}[label={\rm (\arabic*)}]
\item Assumption~\ref{ass:image-condition-strong} for $(V,p)$ holds for every $p \in \mathbb{C}^2$,
\item $\spanC\big(\mathcal{G}_2(V)\big)= {\rm M}_2(\mathbb{C})$.
\end{enumerate}
Let $f:\mathbb{C}^2\to \mathbb{C}^2$ be a polynomial automorphism.
If $uC_f$ is bounded on $V$ for some nonvanishing $u\in \mathcal{O}(\mathbb{C}^2)$, then $f$ is affine.
\end{theorem}

The rest of the paper is organized as follows.
In Section~\ref{sec:local} we introduce the jet filtration and prove the local weighted results.
In Section~\ref{sec:C} we prove the one-variable results.
In Section~\ref{sec:global-affine} we prove the affine-rigidity theorem in several
complex variables.
In Section~\ref{sec:C2} we prove the two-dimensional rigidity theorem.
In Appendix~\ref{app:comparison} we compare the graded image condition with the older kernel formulation.

\section{Local rigidity at periodic points}\label{sec:local}
For $p\in X$ and $n\ge 0$, we put
\[
A_{p,n}:=\mathfrak m_p^n\big/\mathfrak m_p^{n+1},
\qquad
B_{p,n}:=V_{p,n}\big/V_{p,n+1}\subset A_{p,n}.
\]
In other words, $A_{p,n}$ is the space of homogeneous holomorphic $n$-jets at $p$ (see \cite[Section 4]{KPS_NatOpMan93}), and
$B_{p,n}$ is the corresponding subspace induced from $V$.

We note the following formulas for the jet spaces, which immediately follow from the definition:
\begin{align}
    uf^*(\mathfrak m_{f(p)}^n) &\subset \mathfrak m_p^n, \\
    \bigcap_{n\ge 0}V_{p,n} &= \{0\},\label{eq:local-intersection-zero}
\end{align}
In particular, $uf^*$ induces a well-defined linear map $\gr_p^n(uf^*):A_{f(p),n}\to A_{p,n}$ for every $n\ge 0$ introduced in \eqref{eq:local-graded-map}.
Moreover, under Assumption \ref{ass:image-condition} for $(V,f,u,p)$, for infinitely many $n$, the restriction of $\gr_p^n(uf^*)$ induces a linear map $B_{f(p),n}\to B_{p,n}$, which we denote by $\gr_p^n(uC_f)$. 

The next lemma provides the graded action of $uf^*$ on the jet spaces at $f(p)$ and $p$.
\begin{lemma}\label{lem:eigenvalue-local-graded-action}
Let $p$ be a fixed point of $f$, and let $\lambda_1,\dots,\lambda_d$ be the
eigenvalues of ${\rm d}f_p$, counted with algebraic multiplicity.
Let $n\ge 0$.
Then, for every $n_1,\dots,n_d\ge 0$ with $n_1+\cdots+n_d=n$, the complex number
\[
u(p)\lambda_1^{n_1}\cdots\lambda_d^{n_d}
\]
is an eigenvalue of $\gr_p^n(uf^*)$.
In addition, if ${\rm Im}\,\gr_{p}^n(uf^*) \subset V_{p,n}/V_{p,n+1}$ and 
\[
u(p)\lambda_1^{n_1}\cdots\lambda_d^{n_d}\neq 0
\]
hold, then the same complex number is an eigenvalue of $\gr_p^n(uC_f)$.
\end{lemma}

\begin{proof}
Let $L_{p,n}:=\gr_p^n(uf^*):A_{p,n}\to A_{p,n}$.
Choose local holomorphic coordinates at $p$, namely, we fix a holomorphic isomorphism from an open neighborhood of $p$ onto an open neighborhood of $0$ in $\mathbb{C}^d$.
Let $\mathcal{P}_n$ be the space of homogeneous polynomials of degree $n$ in these coordinates.
Then, via this identification, $L_{p,n}$ sends $P\in\mathcal{P}_n$ to $u(p)\,P\bigl({\rm d}f_p(\cdot)\bigr)$.
Hence, for every tuple of non-negative integers $(n_1,\dots,n_d)$,
\[
u(p)\lambda_1^{n_1}\cdots\lambda_d^{n_d}
\]
is an eigenvalue of $L_{p,n}=\gr_p^n(uf^*)$.
Assume $L_{p,n}(A_{p,n})\subset B_{p,n}$.
Let
\[
\mu:=u(p)\lambda_1^{n_1}\cdots\lambda_d^{n_d},
\]
and assume $\mu\neq 0$.
Choose $a\in A_{p,n}\setminus\{0\}$ such that
\[
L_{p,n}a=\mu a.
\]
Since $\mu\neq 0$, we have
\[
a=\mu^{-1}L_{p,n}a\in B_{p,n}.
\]
Therefore, $a$ is an eigenvector of $\gr_p^n(uC_f)$ with eigenvalue $\mu$.
This proves the second assertion.
\end{proof}

Now, we provide the proof of Theorem~\ref{thm:bounded-local}.

\begin{proof}[Proof of Theorem~\ref{thm:bounded-local}]
Set $S:=u_rC_{f^r}$ and $S_{p,n} := \gr_p^n(u_rC_{f^r})$.
Let $\alpha$ be an eigenvalue of ${\rm d}(f^r)_p$.
If $\alpha=0$, then the conclusion is immediate.
Assume from now on that $\alpha\neq 0$.

First, we prove the statement on boundedness.
By Assumption \ref{ass:image-condition} for $(V,f^r,u_r,p)$, there exist infinitely many
$n\ge 1$ such that
\[
\gr_p^n\bigl(u_r(f^r)^*\bigr)(A_{p,n})\subset B_{p,n}.
\]
For each such $n$, Lemma~\ref{lem:eigenvalue-local-graded-action} shows that $u_r(p)\alpha^n$ is an eigenvalue of the induced operator $\gr_p^n(u_rC_{f^r})$.
Since $S_{p,n}$ is a quotient of $u_r C_{f^r}$, we have
\[
\|S_{p,n}\|\le \|S|_{V_{p,n}}\|\le \|S\|.
\]
Hence
\begin{align}
    |u_r(p)|\,|\alpha|^n \le \|S|_{V_{p,n}} \|\le \|S\| \label{eq:eigenvalue-bound}
\end{align}
for infinitely many $n$.
Because $u_r(p)\neq 0$, this is possible only if $|\alpha|\le 1$.

Next, we prove the statement on compactness.
We claim that
\begin{align}\label{eq:local-restriction-norm-zero}
\|S|_{V_{p,n}}\|\longrightarrow 0
\qquad (n\to\infty).
\end{align}
We prove this by contradiction.
Suppose there exist $\varepsilon>0$ and $h_n\in V_{p,n}$ such that
\[
\|h_n\|_V = 1,
\qquad
\|Sh_n\|_V\ge \varepsilon
\]
for infinitely many $n$.
Since $S$ is compact, there exists a subsequence $(Sh_{n_k})$ converging in $V$ to an element $g\in V$.
Fix $m\ge 0$.
For all sufficiently large $k$, we have $n_k\ge m$, thus
\[
h_{n_k}\in V_{p,n_k}\subset V_{p,m}.
\]
Since $f^r(p)=p$, the subspace $V_{p,m}$ is $S$-invariant, so
\[
Sh_{n_k}\in V_{p,m}
\]
for all sufficiently large $k$.
Since $V_{p,m}$ is closed, we obtain $g\in V_{p,m}$.
As $m$ is arbitrary, \eqref{eq:local-intersection-zero} yields $g=0$.
But this contradicts the assumption $\|Sh_{n_k}\|_V\ge \varepsilon$ for all $k$.
Thus, the claim \eqref{eq:local-restriction-norm-zero} holds.
Since the inequality $|u_r(p)|\,|\alpha|^n \le \|S|_{V_{p,n}} \|$ holds for infinitely many $n$ by the same argument as above, we have $|\alpha|<1$.
\end{proof}

Before proving Theorem~\ref{thm:hyper-local}, we prepare the following lemma:
\begin{lemma}\label{lem:hypercyclic-finite-dimension}
    Let $A$ be a linear map on a finite-dimensional complex linear space $W$.
    If $A$ is hypercyclic (resp. supercyclic), then $\dim W = 0$ (resp. $\dim W \le 1$).
\end{lemma}
\begin{proof}
    We fix an inner product $\langle \cdot, \cdot \rangle_W$ of $W$.
    Suppose there exists a hypercyclic vector $x \in W$ for $A$.
    Then, the orbit $\{A^n x:n\ge 0\}$ is dense in $W$.
    Let $y$ be an eigenvector of $A^*$ with eigenvalue $\lambda$.
    Then, the set $\{\langle A^n x,y\rangle_W : n\ge 0\} = {\lambda^n \langle x,y\rangle_W: n\ge 0}$ is dense in $\mathbb{C}$, but it is impossible. 
    Thus, if $A$ is hypercyclic, we have $\dim W = 0$.

    As for the supercyclic case, see \cite{Galaz-Fontes2013}.
\end{proof}

Then, we provide the proof of Theorem~\ref{thm:hyper-local}.
\begin{proof}[Proof of Theorem~\ref{thm:hyper-local}]
Assume that $f$ has a periodic point $p \in X$ of period $r\ge 1$.
For each $n\ge 0$, set
\[
W_n := \bigcap_{j=0}^{r-1} V_{f^j(p),n}.
\]
Then, each $W_n$ is a closed $uC_f$-invariant subspace of finite codimension.
Let $D(A_n):=(D(uC_f)+W_n)/W_n \subset V/W_n$ and define a well-defined linear map $A_n:D(A_n)\to V/W_n$ by $A_n(x+W_n):=uC_f x+W_n$ with $x\in D(uC_f)$.
Since $uC_f$ is supercyclic/hypercyclic, its domain $D(uC_f)$ is dense in $V$.
Hence, $D(A_n)$ is dense in the finite-dimensional space $V/W_n$, so $D(A_n)=V/W_n$.
Therefore, $A_n$ is a linear operator on $V/W_n$.

Let $x$ be a supercyclic (resp. hypercyclic) vector for $uC_f$.
Then, $x\notin W_n$ whenever $V/W_n\neq \{0\}$. In fact the orbit of $x$ would be contained in the proper closed subspace $W_n$, which is impossible.
Therefore, $x+W_n$ is supercyclic (resp. hypercyclic) for $A_n$ whenever $V/W_n\neq\{0\}$.
By Lemma~\ref{lem:hypercyclic-finite-dimension}, we obtain
\begin{align}
    \dim(V/W_n)\le 1
\qquad
(\text{resp. }\dim(V/W_n)=0) \label{eq:dim1or0}
\end{align}
for every $n\ge 0$.

Choose a two-dimensional (resp. one-dimensional) subspace
$E\subset V$.
Since $(E\cap W_n)_{n\ge 0}$ is a decreasing sequence of subspaces of the
finite-dimensional space $E$ and
\[
\bigcap_{n\ge 0}(E\cap W_n)
=
E\cap \bigcap_{n\ge 0}W_n
=
\{0\},
\]
there exists $n$ such that $E\cap W_n=\{0\}$.
Hence, the quotient map $E\to V/W_n$ is injective, namely $\dim(V/W_n)\ge 2$ (resp. $\dim(V/W_n)\ge 1$), contradicting \eqref{eq:dim1or0}.
\end{proof}

% Before proving Theorem~\ref{thm:cyclic-local}, we prepare the following lemma:
% \begin{lemma}\label{lem:local-cyclic-kernel}
% Let $W$ be a complex topological vector space and let $S$ be a cyclic continuous
% linear operator on $W$.
% Then, for every nonzero polynomial $P\in \mathbb{C}[z]$,
% \[
% \dim \Ker P(S')\le \deg P,
% \]
% where $S'$ is the dual operator of $S$ on the continuous dual space $W'$.
% \end{lemma}
% \begin{proof}
% Let $P(z)=a_0+a_1z+\cdots+a_m z^m$ with $a_m\neq 0$, where $m:=\deg P$.
% Suppose that $\ell_1,\dots,\ell_N\in \Ker P(S')$ are linearly independent.
% We claim that $N \le m$.
% Define
% \[
% R:W\to \mathbb{C}^N,
% \qquad
% R(w):=\bigl(\ell_1(w),\dots,\ell_N(w)\bigr).
% \]
% Since the functionals $\ell_1,\dots,\ell_N$ are linearly independent, $R$ is surjective.
% Let $x\in W$ be a cyclic vector for $S$.
% Since $\spanC\{S^n x:n\ge 0\}$ is dense in $W$ and $R$ is continuous and surjective, we see that $\spanC\{R(S^n x):n\ge 0\}=\mathbb{C}^N$.
% For each $i=1,\dots,N$ and each $n\ge 0$, the identity $P(S')\ell_i=0$ provides
% \[
% 0=(P(S')\ell_i)(S^n x)
%  =\sum_{j=0}^m a_j\,\ell_i(S^{n+j}x).
% \]
% Thus, we have $\sum_{j=0}^m a_j\,R(S^{n+j}x)=0$ for all $n\ge 0$.
% Since $a_m\neq 0$, every vector $R(S^{n+m}x)$ is a linear combination of $R(x),R(Sx),\dots,R(S^{m-1}x)$.
% Hence
% \[
% \mathbb{C}^N=\spanC\{R(x),R(Sx),\dots,R(S^{m-1}x)\}.
% \]
% Therefore, $N\le m$.
% \end{proof}
We provide the proof of Theorem~\ref{thm:cyclic-local}.
\begin{proof}[Proof of Theorem~\ref{thm:cyclic-local}]
Let $x\in D_\infty(uC_f)$ be a cyclic vector for $uC_f$.
Fix $r\ge 1$ and $\lambda\in\mathbb{C}$, and set
\[
E_{r,\lambda}:=\mathcal{P}_r(f) \cap u_r^{-1}(\lambda) = \{p\in X : f^r(p)=p,\ u_r(p)=\lambda\}.
\]
Choose distinct points $p_1,\dots,p_N\in E_{r,\lambda}$.
By assumption, the functionals
\[
\delta_{p_1}|_V,\dots,\delta_{p_N}|_V
\]
are linearly independent.
Define
\[
R:V\to \mathbb{C}^N,
\qquad
R(h):=\bigl(h(p_1),\dots,h(p_N)\bigr).
\]
Then $R$ is continuous and surjective.
Since $x$ is cyclic, $\spanC\{(uC_f)^n x:n\ge 0\}$ is dense in $V$.
Hence
\[
\spanC\{R((uC_f)^n x):n\ge 0\}=\mathbb{C}^N.
\]
For each $j=1,\dots,N$ and each $n\ge 0$, we have
\[
\delta_{p_j}\bigl((uC_f)^{n+r}x\bigr)
=
u_r(p_j)\,\delta_{p_j}\bigl((uC_f)^n x\bigr)
=
\lambda\,\delta_{p_j}\bigl((uC_f)^n x\bigr).
\]
Therefore, for any $n \ge 0$, we have
\[
R((uC_f)^{n+r}x)=\lambda\,R((uC_f)^n x).
\]
Thus, every vector $R((uC_f)^n x)$ is a linear combination of $R(x),\,R(uC_fx),\,\dots,\,R((uC_f)^{r-1}x)$.
Hence
\[
\mathbb{C}^N
=
\spanC\{R(x),R(uC_fx),\dots,R((uC_f)^{r-1}x)\},
\]
so $N\le r$.
Therefore, we conclude that $\#E_{r,\lambda}\le r$.
\end{proof}

\section{\texorpdfstring{Entire functions on $\mathbb{C}$}{Entire functions on C}}\label{sec:C}

The one-dimensional case is especially simple since Assumption \ref{ass:image-condition-strong} is automatic as follows:
\begin{lemma}\label{lem:one-dim-graded}
Assume that $\dim_{\mathbb C}X=1$ and that $V$ is infinite-dimensional.
Then, Assumption \ref{ass:image-condition-strong} for $(V,p)$ holds for every $p\in X$.
\end{lemma}
\begin{proof}
Fix $p\in X$.
Since $\dim_{\mathbb C}X=1$, the space $A_{p,n}$ is one-dimensional for every
$n\ge 0$.
Thus, $B_{p,n}$ is either $\{0\}$ or $A_{p,n}$.
Therefore, we have $B_{p,n}=\{0\}$ if and only if $V_{p,n}=V_{p,n+1}$.

Assume that $V_{p,n}=V_{p,n+1}$ for all $n\ge N+1$ for some $N \ge 1$.
By \eqref{eq:local-intersection-zero}, we have $V_{p,N+1}=\{0\}$.
Since $V_{p,N+1}$ is the kernel of the map 
\[
V\longrightarrow \mathbb{C}^{N+1},
\qquad
h\longmapsto \bigl(h(p),h'(p),\dots,h^{(N)}(p)\bigr)
\]
, the dimension of $V$ must be at most $N+1$, which is impossible since $V$ is infinite-dimensional.
Hence, $V_{p,n}\neq V_{p,n+1}$ for infinitely many $n$.
\end{proof}

Now, we prove Theorem~\ref{thm:repelling-fixed-point-strong}.
\begin{proof}[Proof of Theorem~\ref{thm:repelling-fixed-point-strong}]
Put $T:=uC_f$ and fix a point $p\in X$ with $f(p)=p$.
Choose a local coordinate $z$ near $p$ such that $z(p)=0$.
If $u(p)\neq 0$, then the conclusion follows from Theorem~\ref{thm:bounded-local}.
We may assume that $u(p)=0$.
Let $m:=\ord_p(u)\ge 1$, and write
\[
u(z)=u_m z^m+O(z^{m+1}),
\qquad u_m\neq 0.
\]
Suppose that $|f'(p)|>1$.
We denote $B_{p,n} := V_{p,n}/V_{p,n+1}$.
Define
\[
\phi_n:B_{p,n}\to\mathbb{C};~
\phi_n(h + V_{p, n+1}) \mapsto \frac{h^{(n)}(p)}{n!}.
\]
Since $V$ is infinite-dimensional, there exists $n\ge 0$ such that $B_{p,n}\neq 0$.
We note that the map $\phi_n$ is an isomorphism.
Since 
\[
f(z)=f'(p)z+O(z^2).
\]
the operator $T$ induces linear maps $[T]_n:B_{p,n}\to B_{p,n+m}$
satisfying
\begin{align}
  \phi_{n+m}\circ [T]_n
=
f'(p)^n u_m\,\phi_n.  
\label{eq:commutativity}
\end{align}
Iterating the previous identity \eqref{eq:commutativity}, we have
\[
\phi_{n+km}\circ [T]_{n+(k-1)m}\circ \cdots \circ [T]_n
=
u_m^k\,f'(p)^{kn+m k(k-1)/2}\,\phi_n.
\]
Hence
\[
\left\|
\phi_{n+km}\circ [T]_{n+(k-1)m}\circ \cdots \circ [T]_n\circ \phi_n^{-1}
\right\|
=
|u_m|^k\,|f'(p)|^{kn + mk(k-1)/2}.
\]
Fix sufficiently small $R>0$.
Since the inclusion $V\hookrightarrow \mathcal{O}(X)$ is continuous, there exists
$C_R>0$ such that
\[
\sup_{|z|\le R}|h(z)|\le C_R\|h\|_V
\]
for any $h\in V$.
By Cauchy's estimate, we have
\[
\left|\frac{h^{(j)}(p)}{j!}\right|\le C_R R^{-j}\|h\|_V
\]
for any $h\in V,\ j\ge 0$.
Therefore, we obtain
\begin{align}
    \|\phi_j\|\le C_R R^{-j}
\end{align}
for any $j \ge 0$.
By the boundedness of $T$ on $V$, we have 
\begin{align}
    \|[T]_{n+jm}\|\le \|T\|
\end{align}
for arbitrary $j \ge 0$.
Therefore
\[
|u_m|^k\,|f'(p)|^{kn+ mk(k-1)/2}
\le
\|\phi_{n+km}\|\,\|T\|^k\,\|\phi_n^{-1}\|
\le
C_R R^{-(n+km)}\|T\|^k\,\|\phi_n^{-1}\|.
\]
The right-hand side grows at most exponentially in $k$, whereas the left-hand side
grows like
\[
|f'(p)|^{mk^2/2}.
\]
This is impossible since $m \ge 1$ and $|f'(p)|>1$.
Therefore, we have $|f'(p)|\le 1$.
\end{proof}
\begin{remark}
    If we generalize this argument to the higher-dimensional case, it would provide an alternative proof of Theorem~\ref{thm:bounded-local}.
\end{remark}

Now, we prove Corollary~\ref{cor:repelling-fixed-point-strong-entire}.
\begin{proof}[Proof of Corollary~\ref{cor:repelling-fixed-point-strong-entire}]
Suppose that $f$ is non-affine.
By classical one-variable holomorphic dynamics, $f$ has a repelling periodic point
$p$ of some period $r\ge 1$; see, for example, \cite[Theorem~1.20]{Sch10}.
Set
\[
g:=f^r,
\qquad
v:=u_r=\prod_{j=0}^{r-1}u\circ f^j.
\]
Then, $p$ is a fixed point of $g$ and $vC_g=(uC_f)^r$ is bounded on $V$.
Since $p$ is a repelling periodic point of period $r$ for $f$, we have
\[
|g'(p)|=|(f^r)'(p)|>1.
\]
This contradicts Theorem~\ref{thm:repelling-fixed-point-strong} applied to $g$ and $v$.
Therefore, $f$ must be affine.
Write $f(z)=az+b$ for some $a, b \in \mathbb{C}$.
If $|a|>1$, then $f$ has the fixed point $p=\frac{b}{1-a}$ and $|f'(p)|=|a|>1$, which contradicts Theorem~\ref{thm:repelling-fixed-point-strong} again.
Therefore, $|a|\le 1$.
\end{proof}

\medskip
We obtain a necessary condition for (super, hyper)cyclicity for a general topological vector space $V$ of entire functions:
\begin{proposition}\label{prop:hyper-C}
Let $V\subset \mathcal{O}(\mathbb{C})$ be a general topological vector space with continuous inclusion.
Let $f:\mathbb{C}\to\mathbb{C}$ be entire and let $u\in \mathcal{O}(\mathbb{C})$.
Assume one of the following conditions:
\begin{enumerate}[label={\rm (\arabic*)}]
    \item $\dim V \ge 1$ and $uC_f$ is hypercyclic on $V$, \label{eq:u-general-hypercyclic}
    \item $\dim V \ge 2$ and $uC_f$ is supercyclic on $V$. \label{eq:u-general-supercyclic}
\end{enumerate}
Then, there exists $b\in \mathbb{C}\setminus\{0\}$ such that $f(z)=z+b$.
\end{proposition}
\begin{proof}
By Theorem~\ref{thm:hyper-local}, the map $f$ has no periodic points.
A non-affine entire map has periodic points by \cite{Sch10}.
Thus, $f$ must be affine.
If $f(z)=az+b$ with $a\neq 1$, then $f$ has a fixed point.
So $a=1$.
Finally $b\neq 0$ since the identity map also has fixed points.
\end{proof}

\begin{proposition}\label{prop:cyclicity-unweighted}
Let $V\subset \mathcal{O}(\mathbb{C})$ be a general topological vector space with continuous inclusion.
Let $f:\mathbb{C}\to\mathbb{C}$ be entire.
Assume that $\{\delta_p|_V : p \in \mathbb{C}\}$ is linearly independent and that the unweighted composition operator $C_f$ is cyclic.
Then, there exists $a, b\in \mathbb{C}$ such that $f(z)=az+b$.
\end{proposition}
\begin{proof}
Suppose that $f$ is non-affine.
First assume that $f$ is a polynomial of degree $d\ge 2$.
By Theorem~\ref{thm:cyclic-local} in the unweighted case, we have $\#\mathcal{P}_1(f)\le 1$.
Thus, $f$ has a unique fixed point $p$.
Since $f(z)-z$ is a polynomial of degree $d$, this implies
\[
f(z)-z=c(z-p)^d
\]
for some $c\neq 0$ and some $p\in\mathbb{C}$.
Since $f^2(z)-z=2c(z-p)^d+O((z-p)^{d+1})$, the point $p$ is a zero of $f^2(z)-z$ of multiplicity exactly $d$.
Since the degree of $f^2-z$ is greater than $d$, there exists a point $q\neq p$ such that $f^2(q)=q$.
Because $p$ is the unique fixed point, $q$ is not fixed; hence $q$ has exact period $2$.
Thus, $\mathcal{P}_2(f)$ contains the three distinct points $p$, $q$, and $f(q)$, which contradicts $\#\mathcal{P}_2(f)\le 2$ of Theorem~\ref{thm:cyclic-local}.
Next assume that $f$ is transcendental entire.
By \cite{Rosenbloom1952} (see also \cite{Bergweiler01081991}), $\mathcal{P}_2(f)$ is infinite, contradicting Theorem~\ref{thm:cyclic-local}.
Therefore, $f$ must be affine.
\end{proof}

\begin{remark}
Proposition~\ref{prop:hyper-C} gives only a necessary condition, not a characterization.
On the space $\mathcal{O}(\mathbb{C})$, Birkhoff's classical theorem shows that the translation operator $C_f$ with $f(z)=z+b$ is hypercyclic if and only if $b\neq 0$; see \cite{Birkhoff1929}.
Thus, Proposition~\ref{prop:hyper-C} is sharp at the level of the symbol for the full space of entire functions.
For concrete Banach or Hilbert spaces of entire functions, however, cyclicity and hypercyclicity depend strongly on the function space and on the weight.
For instance, Carroll and Gilmore proved that weighted composition operators on Fock spaces are never supercyclic, hence in particular never hypercyclic \cite{CARROLL2021125234}.
On the other hand, cyclicity may still occur in specific spaces:
Bayart and Tapia-Garc\'ia obtained a full characterization of cyclic composition operators on the Fock space \cite{BayartTapia2024}, while Hai, Noor, and Severiano proved that on the Paley--Wiener space $B^2_\sigma$, a bounded composition operator $C_\phi$ is cyclic precisely when $\phi(z)=z+b$ with either $b\in\mathbb{C}\setminus\mathbb{R}$ or $b\in\mathbb{R}$ and $0<|b|\le \pi/\sigma$ \cite{HaiNoorSeveriano2025}.
See also \cite{Bes2013,Bes2014} for broader results on the dynamics of composition and weighted composition operators on spaces of holomorphic functions.
\end{remark}

\section{Higher-dimensional affine rigidity}
\label{sec:global-affine}
Here, we provide the proof of Theorem~\ref{thm:global-affine}.
First, we prepare the following lemma, which is a key step in the proof.
We denote the real part (resp. imaginary part) of a complex number $z$ by $\Re z$ (resp. $\Im z$).
\begin{lemma}
Let $d \ge 2$, and let $f \colon \mathbb{C}^d \to \mathbb{C}^d$ be a non-affine entire map (equivalently, $f$ is not affine). Then, there exist a real number $a \in (0,1)$, a matrix $U \in \mathrm{SU}(d)$, and a point $p \in \mathbb{C}^d$ such that
\[
(f \circ (aU))(p)=p,
\]
and the derivative $D(f \circ (aU))(p)$ has an eigenvalue of modulus strictly greater than $1$.
\end{lemma}

\begin{proof}
We denote the inner product of $\mathbb{C}^d$ by $\langle \cdot,\cdot\rangle$, which is linear in the first variable and with the associated Euclidean norm $\|\cdot\|$.

For unit vectors $v$ and $w$, let $\varphi_{v,w}(\zeta) := \langle f(\zeta v), w \rangle$ and define 
\[M_{v,w}(r):=\max_{|\zeta|=r} |\varphi_{v,w}(\zeta)|.\]
By Hadamard's three-circles theorem \cite[Theorem~3.13]{Conway1978}, $\log M_{v,w}(e^s)$ is convex as a function of $s$. 
For $r>0$, define
\[
M(r):=\sup_{\|v\|=1,\ \|w\|=1} M_{v,w}(r) = \max_{\|z\|=r}\|f(z)\|.
\]
Since $f$ is not affine, there exist unit vectors $v, w\in \mathbb{C}^d$ such that $\langle f(\zeta v), w \rangle$ is a non-affine entire function of the variable $\zeta$.
Thus, $M(r)/r$ is unbounded as $r\to\infty$.
In fact, if $M(r)=O(r)$ for large $r$, then $|\varphi_{v,w}(\zeta)|=O(|\zeta|)$ as $|\zeta|\to\infty$.
By the Cauchy integral formula, $\varphi_{v,w}^{(n)}(0)=0$ for all $n\ge 2$, namely, $\varphi_{v,w}$ must be affine, which is a contradiction.

Let $H(s):=\log (M(e^s)/e^s)$.
Since $H$ is convex and unbounded above, we see that $H$ is differentiable almost everywhere and there exists $s$ such that $H(s)>0$ and $H'(s)>0$.
Set 
\[
r:=e^s\quad\text{and}\quad \eta:=1 + H'(s)=\frac{rM'(r)}{M(r)}.
\]
Choose $q\in \mathbb{C}^d$ with $\|q\|=r$ and $\|f(q)\|=M(r)$, and set
\[
p:=f(q),\qquad B:=df_q.
\]
We note that $p \neq 0$ since $M(r)>r>0$.
Let 
\[
\Phi(z):=\|f(z)\|^2.
\]
Since the point $q$ is a maximum of $\Phi$ on the sphere $S_r:=\{z\in\mathbb{C}^d:\|z\|=r\}$, the derivative $d\Phi_q$ vanishes on the tangent space $T_qS_r$ of $S_r$ at $q$.
In other words, for any $\xi \in \mathbb{C}^d$ with $\Re\langle \xi,q\rangle=0$, we have
\[
0=d\Phi_q(\xi)=2\Re\langle B\xi,p\rangle,
\]
Thus, there exists $\lambda\in\mathbb{R}$ such that
\begin{align}
B^*p=\lambda q.
\end{align}
Let $u:=\frac{q}{r}$ and $m(t):=\|f(tu)\|$ for $t\ge 0$.
Then, we see that both $m(r)=M(r)$ and $m'(r)=M'(r)$ hold since $m(t)\le M(t)$ for all $t>0$ and both $M$ and $m$ are differentiable at $r$.
Therefore, we have
\[
2M(r)M'(r)
=(m^2)'(r)
=2\Re\langle Bu,p\rangle
=\frac{2}{r}\Re\langle Bq,p\rangle
=2\lambda r,
\]
namely
\[
\lambda=\frac{M(r)M'(r)}{r}.
\]
Let
\[
a:=\frac{r}{M(r)}\in(0,1).
\]
Since $\|ap\|=\|q\|=r$, there exists $U\in \mathrm{SU}(d)$ such that
\[
aUp=q.
\]
Define
\[
g:=f\circ(aU).
\]
We claim that $g$ has a fixed point at $p$ and that $Dg(p)$ has an eigenvalue of absolute value greater than $1$.
By direct calculation, we have $g(p)=f(aUp)=f(q)=p$, thus $p$ is a fixed point of $g$.
Let $A:=Dg(p)=aBU$.
Since $a$ is real and $aUp=q$, we have $U^*q=ap$. 
Therefore,
\[
A^*p
=aU^*B^*p
=a\lambda U^*q
=\lambda a^2 p
=\frac{M(r)M'(r)}{r}\cdot \frac{r^2}{M(r)^2}\,p
=\frac{rM'(r)}{M(r)}\,p
=\eta p.
\]
Thus, $\eta>1$ is an eigenvalue of $A^*$. Since the spectrum of $A^*$ is the complex conjugate of the spectrum of $A$, and $\eta$ is real, $\eta$ is also an eigenvalue of $A$.
Therefore, $Dg(p)$ has an eigenvalue of absolute value greater than $1$.
\end{proof}

We provide the proof of Theorem~\ref{thm:global-affine}.
\begin{proof}[Proof of Theorem~\ref{thm:global-affine}]
In the one-dimensional case, the assertion follows from Corollary~\ref{cor:repelling-fixed-point-strong-entire}.
We consider the case of $d \ge 2$.
Suppose that $f$ is not affine.
By the above lemma, there exist $a\in (0,1)$, $U\in \mathrm{SU}(d)$, and $p\in \mathbb{C}^d$ such that $g:=f\circ(aU)$ has a fixed point at $p$ and the derivative $Dg(p)$ has an eigenvalue of absolute value greater than $1$.
By the assumption that there exists a nonvanishing entire function $v$ such that $vC_{aU}$ is bounded on $V$, we have $(v \cdot (u\circ aU))C_g=(vC_{aU})(uC_f)$ is also bounded on $V$.
By Theorem~\ref{thm:bounded-local}, the eigenvalues of $Dg(p)$ have absolute value at most $1$, which is a contradiction.
Therefore, $f$ must be affine.
\end{proof}

\section{\texorpdfstring{Polynomial automorphisms of $\mathbb{C}^2$}{Polynomial automorphisms of C2}}\label{sec:C2}
Here, we provide the proof of Theorem~\ref{thm:C2}.
\begin{proof}[Proof of Theorem~\ref{thm:C2}]
Assume that $f$ is not affine.
By the same argument as in the unweighted proof in \cite[Theorem~1.3]{ISHIKAWA2023109048}, we see that there exists a finite composition $h$ of generalized H\'enon maps and a nonvanishing entire function $v$ such that $vC_h$ is bounded on $V$.
Thus, using \cite[Theorem~3.4]{BS92} as in the proof of \cite[Theorem~1.3]{ISHIKAWA2023109048}, we see that $h$ has a saddle periodic point $p$, which contradicts Theorem~\ref{thm:bounded-local}.
Therefore, $f$ must be affine.
\end{proof}

\begin{remark}
It is natural to expect a higher-dimensional analogue of Theorem~\ref{thm:C2}.
For $d\ge 2$, define
\[
\mathcal{G}_d(V):=
\Big\{A\in {\rm GL}_d(\mathbb{C}) : \exists b\in \mathbb{C}^d,\ \exists w\in \mathcal{O}(\mathbb{C}^d),\
e^wC_{A(\cdot)+b}\text{ bounded on }V\Big\}.
\]
One may conjecture that if Assumption~\ref{ass:image-condition-strong} for $(V,p)$
holds for every $p\in \mathbb{C}^d$ and
\[
\spanC\bigl(\mathcal{G}_d(V)\bigr)={\rm M}_d(\mathbb{C}),
\]
then every bounded weighted composition operator with nonvanishing weight has affine symbol.
For several concrete spaces of entire functions, higher-dimensional affine-symbol results are already known by direct operator-theoretic methods; see, for example,
\cite{DKL17,SS17,IIS20}.
In the weighted Fock setting, strong rigidity results are also available under
additional assumptions such as invertibility or unitarity
\cite{Zhao2014,Zhao2015}.
A general dynamical proof beyond polynomial automorphisms would likely require
new results on the existence of repelling or saddle periodic points in several
complex variables; compare \cite{10.2307/43736735}.
\end{remark}

\bibliographystyle{plain}
\bibliography{reference}

\begin{thebibliography}{10}

\bibitem{ArroussiTong2019}
Hicham Arroussi and Cezhong Tong.
\newblock Weighted composition operators between large {Fock} spaces in several
  complex variables.
\newblock {\em Journal of Functional Analysis}, 277(10):3436--3466, 2019.

\bibitem{BayartTapia2024}
Fr{\'e}d{\'e}ric Bayart and Sebasti{\'a}n Tapia-Garc{\'i}a.
\newblock Cyclicity of composition operators on the fock space.
\newblock {\em Journal of Operator Theory}, 92(2):549--577, 2024.

\bibitem{BS92}
Eric Bedford and John Smillie.
\newblock Polynomial diffeomorphisms of $\mathbb{C}^2$.
\newblock {\em Mathematische Annalen}, 294(1):395--420, 1992.

\bibitem{BeltranMeneu2020}
Mar{\'i}a~J. Beltr{\'a}n-Meneu.
\newblock Dynamics of weighted composition operators on weighted {Banach}
  spaces of entire functions.
\newblock {\em Journal of Mathematical Analysis and Applications},
  492(1):124422, 2020.

\bibitem{BeltranMeneuJorda2021}
Mar{\'i}a~J. Beltr{\'a}n-Meneu and Enrique Jord{\'a}.
\newblock Dynamics of weighted composition operators on spaces of entire
  functions of exponential and infraexponential type.
\newblock {\em Mediterranean Journal of Mathematics}, 18:212, 2021.

\bibitem{Bergweiler01081991}
Walter Bergweiler.
\newblock Periodic points of entire functions: proof of a conjecture of baker.
\newblock {\em Complex Variables, Theory and Application: An International
  Journal}, 17(1-2):57--72, 1991.

\bibitem{Bes2013}
Juan B{\`e}s.
\newblock Dynamics of composition operators with holomorphic symbol.
\newblock {\em Revista de la Real Academia de Ciencias Exactas, F{\'i}sicas y
  Naturales. Serie A. Matem{\'a}ticas}, 107:437--449, 2013.

\bibitem{Bes2014}
Juan B{\`e}s.
\newblock Dynamics of weighted composition operators.
\newblock {\em Complex Analysis and Operator Theory}, 8(1):159--176, 2014.

\bibitem{Birkhoff1929}
G.~D. Birkhoff.
\newblock D{\'e}monstration d'un th{\'e}or{\`e}me {\'e}l{\'e}mentaire sur les
  fonctions enti{\`e}res.
\newblock {\em Comptes Rendus Hebdomadaires des S{\'e}ances de l'Acad{\'e}mie
  des Sciences}, 189:473--475, 1929.

\bibitem{Bonet2022Survey}
Jos{\'e} Bonet.
\newblock Weighted {Banach} spaces of analytic functions with sup-norms and
  operators between them: a survey.
\newblock {\em Revista de la Real Academia de Ciencias Exactas, F{\'i}sicas y
  Naturales. Serie A. Matem{\'a}ticas}, 116(4):184, 2022.

\bibitem{CARROLL2021125234}
Tom Carroll and Clifford Gilmore.
\newblock Weighted composition operators on fock spaces and their dynamics.
\newblock {\em Journal of Mathematical Analysis and Applications},
  502(1):125234, 2021.

\bibitem{CMS}
Brent~J. Carswell, Barbara~D. MacCluer, and Alex Schuster.
\newblock Composition operators on the {F}ock space.
\newblock {\em Acta Scientiarum Mathematicarum}, 69(3-4):871--887, 2003.

\bibitem{Conway1978}
John~B. Conway.
\newblock {\em Functions of One Complex Variable I}, volume~11 of {\em Graduate
  Texts in Mathematics}.
\newblock Springer, 2 edition, 1978.

\bibitem{10.2307/43736735}
John~Erik Forn^^c3^^a6ss and Nessim Sibony.
\newblock Some open problems in higher dimensional complex analysis and complex
  dynamics.
\newblock {\em Publicacions Matem^^c3^^a0tiques}, 45(2):529--547, 2001.

\bibitem{Galaz-Fontes2013}
F.~Galaz-Fontes.
\newblock Another proof for non-supercyclicity in finite dimensional complex
  banach spaces.
\newblock {\em The American Mathematical Monthly}, 120(5):pp. 466--468, 2013.

\bibitem{HaiKhoi2016}
Pham~Viet Hai and Le~Hai Khoi.
\newblock Boundedness and compactness of weighted composition operators on
  {Fock} spaces $\mathcal{F}^{p}(\mathbb{C})$.
\newblock {\em Acta Mathematica Vietnamica}, 41(3):531--537, 2016.

\bibitem{HaiNoorSeveriano2025}
Pham~Viet Hai, Waleed Noor, and Osmar~Reis Severiano.
\newblock Cyclicity of composition operators on the {Paley--Wiener} spaces.
\newblock {\em Comptes Rendus. Math{\'e}matique}, 363:879--886, 2025.

\bibitem{HaiRosenfeld2021}
Pham~Viet Hai and Joel~A. Rosenfeld.
\newblock Weighted composition operators on the mittag-leffler spaces of entire
  functions.
\newblock {\em Complex Analysis and Operator Theory}, 15:17, 2021.

\bibitem{Ueki2007}
Sei ichiro Ueki.
\newblock Weighted composition operator on the {Fock} space.
\newblock {\em Proceedings of the American Mathematical Society},
  135(5):1405--1410, 2007.

\bibitem{Ueki2010}
Sei ichiro Ueki.
\newblock Weighted composition operators on some function spaces of entire
  functions.
\newblock {\em Bulletin of the Belgian Mathematical Society--Simon Stevin},
  17(2):343--353, 2010.

\bibitem{IIS20}
Masahiro Ikeda, Isao Ishikawa, and Yoshihiro Sawano.
\newblock Boundedness of composition operators on reproducing kernel {H}ilbert
  spaces with analytic positive definite functions.
\newblock {\em Journal of Mathematical Analysis and Applications},
  511(1):126048, 2022.

\bibitem{ISHIKAWA2023109048}
Isao Ishikawa.
\newblock {Bounded composition operators on functional quasi-Banach spaces and
  stability of dynamical systems}.
\newblock {\em Advances in Mathematics}, 424:109048, 2023.

\bibitem{KPS_NatOpMan93}
Ivan Kol\'{a}\v{r}, Peter~W. Michor, and Jan Slov\'{a}k.
\newblock {\em Natural operations in differential geometry}.
\newblock Springer-Verlag, Berlin, 1993.

\bibitem{Le2014}
Trieu Le.
\newblock Normal and isometric weighted composition operators on the fock
  space.
\newblock {\em Bulletin of the London Mathematical Society}, 46:847--856, 07
  2014.

\bibitem{DKL17}
Doan Luan, Le~Hai~Khoi, and Trieu Le.
\newblock Composition operators on {H}ilbert spaces of entire functions of
  several variables.
\newblock {\em Integral Equations and Operator Theory}, 88(3):301--330, 2017.

\bibitem{Rosenbloom1952}
P.~C. Rosenbloom.
\newblock The fix-points of entire functions.
\newblock {\em Meddelanden fr^^c3^^a5n Lunds Universitets Matematiska
  Seminarium}, 1952(Tome Supplementaire):186--192, 1952.

\bibitem{Sch10}
Dierk Schleicher.
\newblock Dynamics of entire functions.
\newblock In G.~Gentili, J.~Guenot J, and G.~Patrizio, editors, {\em
  Holomorphic Dynamical Systems}, volume 1998, pages 295--339. Springer,
  Berlin, Heidelberg, Berlin, Heidelberg, 2010.

\bibitem{SS17}
Jan Stochel and Jerzy~Bartlomiej Stochel.
\newblock Composition operators on {H}ilbert spaces of entire functions with
  analytic symbols.
\newblock {\em Journal of Mathematical Analysis and Applications},
  454(2):1019^^e2^^80^^931066, 2017.

\bibitem{TienKhoi2019}
Pham~Trong Tien and Le~Hai Khoi.
\newblock Weighted composition operators between different {Fock} spaces.
\newblock {\em Potential Analysis}, 50:171--195, 2019.

\bibitem{Zhao2014}
Liankuo Zhao.
\newblock Unitary weighted composition operators on the fock space of
  $\mathbb{C}^n$.
\newblock {\em Complex Analysis and Operator Theory}, 8:581--590, 02 2014.

\bibitem{Zhao2015}
Liankuo Zhao.
\newblock Invertible weighted composition operators on the fock space of
  $\mathbb{C}^n$.
\newblock {\em Journal of Function Spaces}, 2015:250358, Jun 2015.

\end{thebibliography}

\appendix
\section{\texorpdfstring{Comparison with the earlier kernel condition}{Comparison with the earlier kernel condition}}
\label{app:comparison}

We compare Assumption~\ref{ass:image-condition} with the kernel condition from \cite{ISHIKAWA2023109048}.
We recall the notation from \cite{ISHIKAWA2023109048}.
We denote the continuous dual of the topological vector space $W$ by $W'$.
For a continuous linear map $T:W_1 \to W_2$ of topological vector spaces, we denote its dual map by $T':W_2' \to W_1'$.
For $n\ge 0$, let
\[
\mathcal{D}_n^{\mathrm{hol}}(\mathbb{C}^d)
:=
\bigoplus_{|\alpha|\le n}\mathbb{C}\,\partial_z^\alpha,
\qquad
 \mathcal{D}^{\mathrm{hol}}(\mathbb{C}^d):=\bigcup_{n\ge 0} \mathcal{D}_n^{\mathrm{hol}}(\mathbb{C}^d).
\]
Let $p\in X$ and let $\phi$ be a local holomorphic coordinate at $p$.
Define
\[
\delta_{p,\phi}(D)(h):=D(h\circ \phi^{-1})(\phi(p))
\qquad
(D\in  \mathcal{D}^{\mathrm{hol}}(\mathbb{C}^d),\ h\in \mathcal{O}(X)).
\]
We write
\[
 \mathcal{D}_n^{\mathrm{hol}}(X)_p:=\delta_{p,\phi}\bigl( \mathcal{D}_n^{\mathrm{hol}}(\mathbb{C}^d)\bigr),
\qquad
\mathcal{D}_{-1}^{\mathrm{hol}}(X)_p:=\{0\}.
\]
Note that $ \mathcal{D}_n^{\mathrm{hol}}(X)_p$ depends only on $p$, not on the choice of $\phi$.
The dual map of $\iota:V\hookrightarrow \mathcal{O}(X)$ induces
\[
\kappa^{n,\mathrm{hol}}_p:
 \mathcal{D}_n^{\mathrm{hol}}(X)_p/\mathcal{D}_{n-1}^{\mathrm{hol}}(X)_p
\longrightarrow
\iota'\!\bigl( \mathcal{D}_n^{\mathrm{hol}}(X)_p\bigr)\Big/\iota'\!\bigl(\mathcal{D}_{n-1}^{\mathrm{hol}}(X)_p\bigr).
\]
Also, the dual of the map $\gr_p^n(uf^*):A_{f(p),n}\to A_{p,n}$ induces
\[\gr^{n,\mathrm{hol}}_p((uf^*)'):
 \mathcal{D}_n^{\mathrm{hol}}(X)_p/\mathcal{D}_{n-1}^{\mathrm{hol}}(X)_p 
\longrightarrow 
 \mathcal{D}_n^{\mathrm{hol}}(X)_{f(p)}/\mathcal{D}_{n-1}^{\mathrm{hol}}(X)_{f(p)}.\]
Here, $\gr^{n,\mathrm{hol}}_p((uf^*)')$ coincides with $\gr^n_{(f^*)'}$ in \cite{ISHIKAWA2023109048} when $u \equiv 1$.

\begin{proposition}\label{prop:local-assumption-kappa}
Let $f:X\to X$ be holomorphic and let
$u\in \mathcal{O}(X)$.
Then, the following are equivalent:
\begin{enumerate}[label={\rm (\arabic*)}]
\item $(V,f,u,p)$ satisfies Assumption~\ref{ass:image-condition}.
\item $\Ker \kappa^{n,\mathrm{hol}}_p\subset
\Ker \gr^{n,\mathrm{hol}}_p((uf^*)')$
holds for infinitely many $n\ge 1$.
\end{enumerate}
\end{proposition}
\begin{proof}
If $u(p)=0$, then $\gr_p^n(uf^*)=0$ for every $n\ge 1$, thus the equivalence is obvious.
We assume that $u(p)\neq 0$.
Consider the well-defined perfect pairing
\[
\langle\cdot,\cdot\rangle_n:
\Bigl( \mathcal{D}_n^{\mathrm{hol}}(X)_p/\mathcal{D}_{n-1}^{\mathrm{hol}}(X)_p\Bigr)\times A_{p,n}\to\mathbb C
\]
defined by
\[
\langle [D],[h]\rangle_n:=D(h).
\]
Then, we have an identification
\[
 \mathcal{D}_n^{\mathrm{hol}}(X)_p/\mathcal{D}_{n-1}^{\mathrm{hol}}(X)_p \cong A_{p,n}'.
\]
The pairing restricted to $V_{p,n}\subset \mathfrak m_p^n$ induces another identification
\[
\iota'\!\bigl( \mathcal{D}_n^{\mathrm{hol}}(X)_p\bigr)\Big/\iota'\!\bigl(\mathcal{D}_{n-1}^{\mathrm{hol}}(X)_p\bigr)
\cong B_{p,n}'.
\]
Under these identifications
\[
\kappa^{n,\mathrm{hol}}_p:
 \mathcal{D}_n^{\mathrm{hol}}(X)_p/\mathcal{D}_{n-1}^{\mathrm{hol}}(X)_p
\longrightarrow
\iota'\!\bigl( \mathcal{D}_n^{\mathrm{hol}}(X)_p\bigr)\Big/\iota'\!\bigl(\mathcal{D}_{n-1}^{\mathrm{hol}}(X)_p\bigr)
\]
coincides with the restriction map
\[
A_{p,n}'\longrightarrow B_{p,n}'.
\]
Therefore, we have
\[
\Ker \kappa^{n,\mathrm{hol}}_p \cong B_{p,n}^{\perp}
:=\{\ell \in A_{p,n}':\ell|_{B_{p,n}}=0\}.
\]
Put
\begin{align*}
L_n &:=\gr_p^n(uf^*):A_{f(p),n}\to A_{p,n},\\
G_n &:=\gr^{n,\mathrm{hol}}_p((uf^*)'):  \mathcal{D}_n^{\mathrm{hol}}(X)_p/\mathcal{D}_{n-1}^{\mathrm{hol}}(X)_p\to  \mathcal{D}_n^{\mathrm{hol}}(X)_{f(p)}/\mathcal{D}_{n-1}^{\mathrm{hol}}(X)_{f(p)}.
\end{align*}
We note that these two maps are dual to each other with respect to $\langle\cdot,\cdot\rangle_n$, namely, $\langle G_n\ell,a\rangle_n=\langle \ell,L_n a\rangle_n$  holds for $\ell\in A_{p,n}'$, $a\in A_{f(p),n}$.

First, assume that
\[
\Ker \kappa^{n,\mathrm{hol}}_p\subset \Ker G_n.
\]
Let $a\in A_{f(p),n}$ and let $\ell\in B_{p,n}^{\perp}=\Ker \kappa^{n,\mathrm{hol}}_p$.
Then, $G_n\ell=0$, so
\[
\ell(L_n a)=\langle \ell,L_n a\rangle_n
=\langle G_n\ell,a\rangle_n=0.
\]
Thus, $L_n a\in (B_{p,n}^{\perp})^{\perp} = B_{p,n}$.
Since $a$ is arbitrary, we have $L_n(A_{f(p),n})\subset B_{p,n}$, namely, Assumption~\ref{ass:image-condition} for $(V,f,u,p)$.

Conversely, assume that
\[
L_n(A_{f(p),n})\subset B_{p,n}.
\]
Let $\ell\in \Ker \kappa^{n,\mathrm{hol}}_p=B_{p,n}^{\perp}$ and let $a\in A_{f(p),n}$.
Then, $L_n a\in B_{p,n}$, hence
\[
\langle G_n\ell,a\rangle_n
=\langle \ell,L_n a\rangle_n=0.
\]
Since $a\in A_{f(p),n}$ is arbitrary, we obtain $G_n\ell=0$.
Therefore
\[
\Ker \kappa^{n,\mathrm{hol}}_p\subset \Ker G_n.
\]
\end{proof}

\end{document}